\documentclass[11pt,a4paper]{article}

\usepackage[T1]{fontenc}
\usepackage[utf8]{inputenc}
\usepackage[english]{babel}

\usepackage[a4paper,margin=2.8cm]{geometry}
\usepackage{amsmath,amssymb,amsthm,mathtools}
\usepackage{mathrsfs}
\usepackage{graphicx}
\usepackage{xcolor}
\usepackage{hyperref}
\usepackage[nameinlink,noabbrev]{cleveref}
\usepackage{enumitem}
\usepackage{microtype}

\hypersetup{
	colorlinks=true,
	linkcolor=blue,
	citecolor=blue,
	urlcolor=blue,
	pdftitle={Asymptotics of Parking Search in Hyperfractal Networks},
	pdfauthor={Geoffrey Deperle, Christine Fricker, Philippe Jacquet, Bernard Mans, Alessia Rigonat}
}

\title{\textbf{Asymptotics of Parking Search in Hyperfractal Networks}}

\author{
	Geoffrey Deperle\thanks{INRIA, Palaiseau, France. \texttt{geoffrey.deperle@inria.fr}}
	\and
	Christine Fricker\thanks{INRIA, Paris, France. \texttt{christine.fricker@inria.fr}}
	\and
	Philippe Jacquet\thanks{INRIA, Palaiseau, France. \texttt{philippe.jacquet@inria.fr}}
	\and
	Bernard Mans\thanks{Macquarie University, Australia. \texttt{bernard.mans@mq.edu.au}}
	\and
	Alessia Rigonat\thanks{INRIA, Paris, France. \texttt{alessia.rigonat@inria.fr}}
}

\date{}

\def\E{\mathbb{E}}
\def\Var{\mathbb{V}ar}
\def\CS{{\cal S}}
\def\km{k_{\max}}
\def\tgamma{\tilde{\gamma}}
\def\bT{{\bf T}}
\newcommand{\ind}[1]{\mathbf{1}_{\{#1\}}}

\newtheorem{theorem}{Theorem}[section]
\newtheorem{lemma}[theorem]{Lemma}
\newtheorem{remark}[theorem]{Remark}
\newtheorem*{theorem*}{Theorem}

\begin{document}
	
	\maketitle
	
	\begin{abstract}
		We study the asymptotic behaviour of the distance to the first available parking slot in a recursive Manhattan street network endowed with a hyperfractal intensity structure, where slot-release events occur according to Poisson processes along the streets.
		
		We establish, by analysing the associated self-similar harmonic sums via Mellin-transform asymptotics \cite{mellin}, a power-law decay of the expected distance as the total intensity grows, with exponent equal to the inverse of the hyperfractal dimension. In particular, the scaling exponent depends only on the large-scale geometry of the network.
		
		We further prove that this exponent is robust under random multiplicative modulations of the street intensities: mild stochastic heterogeneity affects only the multiplicative constant. Similar scaling behaviour holds for the variance, the number of turns before parking, and for a jump-over variant of the search strategy.
	\end{abstract}
	\section{Introduction}
	\subsection{Motivation, Problem setting and related work}
	This work is motivated by free-floating (FF) car-sharing systems, which have developed rapidly in many cities as an alternative to private car ownership (see \cite{toy2025zero} and references therein). Unlike station-based systems, where vehicles must be returned to fixed locations, free-floating services allow cars to be parked anywhere within a designated service area (see \cite[Figure 1]{carsharingGermany}). In such systems, finding a parking space becomes a central operational issue.
	
	In~\cite{FMR-FF2025}, the problem of finding resources, both cars and parking spaces, is analyzed for FF car-sharing. In contrast with station-based models such as bike-sharing systems (see~\cite{fricker2016incentives}), the service area is partitioned into large-capacity zones where private cars, much more numerous than FF vehicles, shape parking dynamics. In this setting, the number of available parking spaces is modeled as a random quantity depending primarily on private-car arrivals and departures.
	
	To analyse parking availability, \cite{FMR-FF2025} introduces a homogeneous mean-field model in which the service area is partitioned into $N$ large-capacity cells. The parameter $N$ acts as a scaling parameter, and parking dynamics are driven by private-car arrivals and departures.
	
	In~\cite{FMR-FF2025}, we show that a phase transition separates two regimes: an {\em overloaded} regime where the number of available slots per zone is of order $1$, and an {\em underloaded} regime where the number of available slots per zone is of order $N$.  More precisely, in the overloaded case,   the conclusion is that, when $N$ becomes large (as cell capacities tend to $+\infty$), on the original time scale, the number of available parking places  (of order $1$) is random with a geometric distribution and independent of the number of FF cars. The parameter of the geometric distribution depends only on the private-car and city parameters. It is $\beta C/\alpha$, where $CN$ is the capacity of one cell, $\alpha N$ the arrival rate of private cars in a cell and $1/\beta$ the mean parking duration for private cars. Then  $\beta C/\alpha$ is the probability of finding an available slot to park within a cell, for any car, whether private or FF. However, the analysis assumes spatial homogeneity and does not incorporate the hierarchical structure of urban streets.
	
	The model presented in~\cite{FMR-FF2025} explains why the operator of a FF car-sharing system has no
	influence on the availability of parking spaces and cannot help users find an available space. But the model ignores the structure of the streets in a city, with different average parking times between large and small streets. The analysis of a model designed for this purpose that could provide an estimate of the time to find a parking slot. This is the aim of the paper.
	
	\subsection{Overview of the results}
	The starting point of the analysis is a hyperfractal construction of a Manhattan street network, where slot-release intensities are distributed across hierarchical street levels. We first describe this geometric structure and its characteristic exponent $d_F>2$, which acts as an effective dimension for the search problem.
	
	The search process is then reformulated as a recursive progression along streets of decreasing depth, leading to a self-similar harmonic representation of $\E[D(\lambda)]$. In the canonical ``jumpless'' strategy, this structure yields
	\begin{theorem*}[Scaling law]
		Let $D(\lambda)$ denote the distance to the first available parking slot in the hyperfractal Manhattan model (defined in  Section~3) with total pop-up intensity $\lambda$, and let $d_F$ be the associated hyperfractal dimension. Then, as $\lambda \to \infty$,
		\[
		\mathbb{E}[D(\lambda)] \underset{\lambda \to +\infty}{\sim} C\,\lambda^{-1/d_F},
		\qquad
		\mathrm{Var}(D(\lambda)) \underset{\lambda \to +\infty}{\sim} C'\,\lambda^{-2/d_F},
		\]
		for some constants $C,C'>0$.
	\end{theorem*}
	showing that the scaling is governed by the hyperfractal geometry through the exponent~$d_F$.
	
	Mellin-transform asymptotics allow us to extract this behaviour from the harmonic structure and also provide the corresponding scaling of the variance. We then show that the exponent $1/d_F$ is stable under random multiplicative modulation of the level intensities, highlighting the dominant role of large-scale geometry over local fluctuations.
	
	Beyond the mean distance, we analyse the distribution of the number of turns before parking and prove a logarithmic growth in~$\lambda$ with bounded log-periodic fluctuations. Finally, we consider a jump-over strategy and indicate how the same analytic framework extends beyond the Poisson Manhattan construction to more general hyperfractal street distributions.
	
	\section{Parking strategy and algorithm}
	We study the problem of finding a free parking slot in a given network of city streets. We assume that a car seeking a parking space performs a random path formed by concatenating several street segments $\CS_1, \CS_2,\ldots,\CS_k$. When moving at constant speed $v_i$ on the segment $\CS_i$, free parking spaces appear ahead of the car as a Poisson process with intensity $\lambda_i$. To simplify, we assume that the driving speed $v$ is independent of the segment on which the car is progressing and that $\lambda_i$ is a linear density rather than a time intensity.
	
	Let $D(\CS_1,\ldots,\CS_k)$ be the average driven distance before finding a free parking slot, given the sequence of path segments. Let $|\CS|$ denote the length of a segment $\CS$. With probability $1-e^{-\lambda_1|\CS_1|}$ the first available slot appears on segment $\CS_1$. If $n$ is the number of free slots that appear during the traversal of segment $\CS_1$ the average driven distance is $|\CS_1|/(1+n)$. 

	\begin{equation}
		D(\CS_1,\CS_2,\ldots,\CS_k)=\sum_{n>0}\frac{|\CS_1|}{n+1}\frac{(\lambda_1|\CS_1|)^n}{n!}e^{-\lambda_1|\CS_1|}+e^{-\lambda_1|\CS_1|}\left(|\CS_1|+D(\CS_2,\ldots,\CS_k)\right).
		\label{eq1}
	\end{equation}
	
	Then, by straightforward calculation, 
	\begin{align}\label{eq2}
		D(\CS_1,\CS_2,\ldots,\CS_k)=\frac{1-e^{-\lambda_1|\CS_1|}}{\lambda_1}+e^{-\lambda_1|\CS_1|} D(\CS_2,\ldots,\CS_k). 
	\end{align}
	which yields, by induction,
	
	\begin{equation}\label{eq:meanJumpless}
		D(\CS_1,\ldots,\CS_k)=\sum_{i=1}^{k}\frac{1-e^{-\lambda_i|\CS_i|}}{\lambda_i}\prod_{j<i}e^{-\lambda_j |\CS_j|}.
	\end{equation}
	
	\begin{remark}
		The quantity defined in~\eqref{eq:meanJumpless} corresponds to a prescribed sequence of segments.
		One may also consider an optimal strategy, where the next segment is chosen among those accessible from the current one. In that case, the expected distance satisfies a dynamic programming relation of the form
		\[
		V(\CS_1)=\frac{1-e^{-\lambda_1|\CS_1|}}{\lambda_1}
		+e^{-\lambda_1|\CS_1|}\min_{\CS_2} V(\CS_2),
		\]
		where $\CS_2$ ranges over the segments reachable from $\CS_1$.
	\end{remark}

	\section{Hyperfractal geometry}
	The motivation for the hyperfractal geometry comes from the well-known hierarchical organisation of urban street networks, where major roads, arterial streets and residential streets play different structural roles. \cite{batty1994fractal}
	
	The first observation is that in a city the busiest streets are interleaved with less busy ones, typically narrower streets. In a sense, lower-traffic streets form a dense pattern at any scale on a city map. The second observation is that the street pattern is often reminiscent of a fractal structure. \cite{hyperfractal}
	
	In our description, we assume that the free-slot pop-up rates $\lambda_i$ are distributed according to a hyperfractal measure on a Manhattan street network. Indeed, we assume that the streets are of various types: high-traffic roads, arterial streets, regular streets, residential streets. 
	
	As most drivers are well aware, streets, regardless of their status, are constantly congested with parked cars. The only thing that may vary depending on the status of the street is the frequency with which free parking spaces appear or disappear. On a residential street, most parking spaces are occupied by residents who may not commute frequently. At the other end of the spectrum, busy streets border major shopping centres, and parking spaces can actually experience a much higher turnover rate.

	We consider a city map with a Manhattan layout with North-South streets and East-West streets within the unit square. The streets are the support of a measure $\mu$. In a hyperfractal structure, streets are ranked according to a depth index. The depth index varies between 0 and $\km$. 
	
	\subsection{A deterministic Manhattan hyperfractal construction}
	We describe a deterministic recursive construction of a measure $\mu$ supported on a Manhattan-type street grid.
	
	At level $0$, the unit square contains the central cross formed by one vertical and one horizontal street; each of these two streets carries mass $p/2$ with $p \in (0,1)$, uniformly distributed along the segment. 
	
	At level $k \ge 1$, we add $2^k$ vertical streets and $2^k$ horizontal streets, placed at equally spaced dyadic positions parallel to the central cross. Each street at level $k$ carries mass 
	$\frac{p}{2}\left(\frac{q}{2}\right)^k$ with $q = 1-p$ uniformly distributed along its length.
	
	Thus, the number of streets grows geometrically while their individual mass decreases geometrically. When the construction is iterated up to depth $k_{\max}$, we obtain a weighted Manhattan grid; when $k_{\max}=\infty$, this yields the hyperfractal Manhattan model (Figure~\ref{fig0}).
	\begin{figure}[!h]
		\centering
		\includegraphics[width=3.1cm]{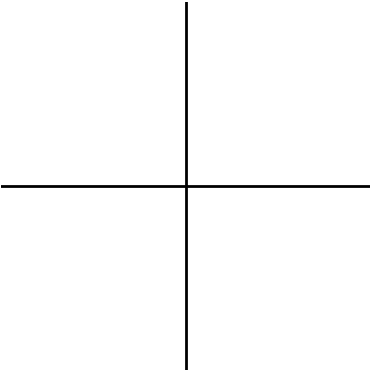}
		\includegraphics[width=3.1cm]{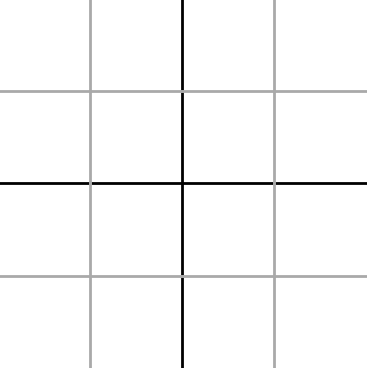}
		\includegraphics[width=3.1cm]{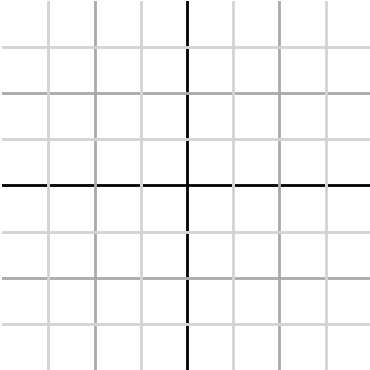}
		\includegraphics[width=3.1cm]{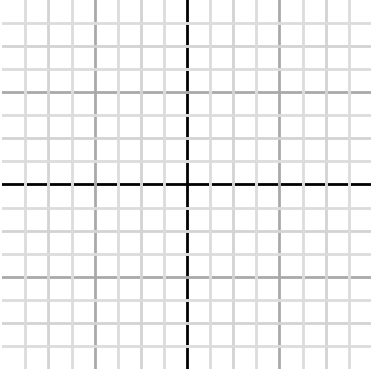}
		\caption{Recursive construction of the Manhattan model}
		\label{fig0}\end{figure}
	
	The result of the infinite process is a probability measure on the unit square whose support is the set of streets. The measure has a fractal nature since it is defined via a repeated pattern on the successive quadrants. Since each quadrant with half the side length of the original square receives a fraction $q/4$ of the total mass, the fractal dimension satisfies the identity
	\[
	\left(\frac{1}{2}\right)^{d_F}=\frac{q}{4} \quad \text{or, equivalently,} \ \ \ \ \ \ d_F=\frac{\log(4/q)}{\log 2}.
	\]
	We observe that the quantity $d_F$ is greater than the Euclidean dimension 2 of the city map. We therefore call $\mu$ a hyperfractal measure. Indeed, fractal objects are always of dimension smaller than the Euclidean dimension. This is mainly because  fractal objects studied in the literature were originally subsets of $\mathbb{R}^2$. But this may no longer be true when subsets are replaced by measures.
	
	As a consequence of the recursive construction, the density of $\mu$ on a street of rank $k$ is uniform and equal to $\mu_k=\frac{p}{2}(\frac{q}{2})^k$.
	
	\subsection{A Poisson Manhattan hyperfractal measure}\label{sec:PoissonManhattan}
	The $x$-coordinates of the north-south streets are distributed according to a Poisson process with rate based on the depth index $k$, i.e. $2^k$. Similarly, the coordinates of east-west streets are generated by an independent Poisson process with the same depth-dependent rate $2^k$. Figure~\ref{fig2} shows an example of Poisson street distribution.
	
	\begin{figure}[h]
		\centering
		\includegraphics[width=3.1cm]{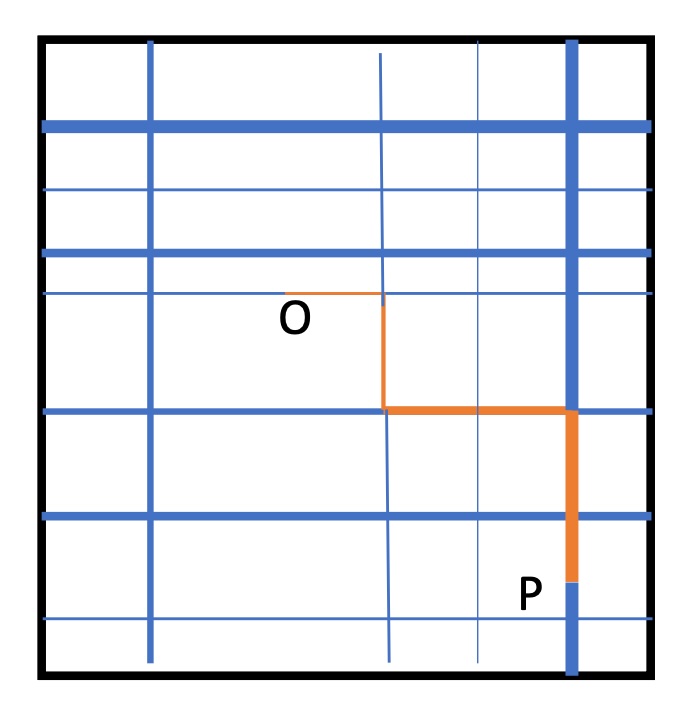}
		\caption{Poisson generation of streets, color intensity indicates the intensity of pop-up rates. Between point O and point P a non backtracking East/South search}
		\label{fig2}\end{figure}
	
	In a street of depth $k$ the parking pop-up rate is arbitrarily set to $\lambda_k=\mu_k\lambda$ (for the binary hyperfractal structure~\cite{hyperfractal}) where $\lambda$ is the total parking pop-up rate aggregated over the city. The hyperfractal dimension is $d_F=\frac{\log(4/q)}{\log 2}$ and is larger than 2. In the binary hyperfractal map, the streets are organized in a Manhattan pattern with an average of $2^{\km}$ North-South streets and $2^{\km}$ East-West streets. We observe that the streets with lower popup rates tend to be dense in the unit square when $\km$ increases.

	\section{The analysis}
	\subsection{The steady jumpless model}
	Since the streets are distributed according to a Poisson process (see Section~\ref{sec:PoissonManhattan}), we denote by $\alpha_i$ the rate parameter of the segment length on $\CS_i$, so that $|\CS_i|$ is exponentially distributed with mean $1/\alpha_i$. Integrating on segment lengths $|\CS_1|,\ldots,|\CS_k|$ in~\eqref{eq:meanJumpless}, it holds that
	\begin{align}\label{mean_Jumpless2}
		\E(D(\CS_1,\ldots,\CS_k))=\sum_{i=1}^{k}\frac{1/\alpha_i}{1+\lambda_i/\alpha_i}\prod_{j<i}\frac{1}{1+\lambda_j/\alpha_j}.  
	\end{align}
	
	We analyse a route in which a car starts on a street at level $\km$ and, if unsuccessful, turns onto the first street at level $\km-1$. The car will turn onto streets of strictly decreasing depth, $\km$, $\km-1$, $\km-2$, {\it etc},
	until it finds a free slot. In this model, we require that the depth of the segment be strictly decremented by one unit at each turn. A non backtracking search is when the car always increase its distance with the origin, by alternating East and South turns for example as depicted by figure~\ref{fig2}. That way the segments are independent. However this strategy is sub-optimal since a street with a much lower depth may be crossed before turning onto a street with a depth decreased by exactly one unit. Allowing the depth to jump by more than one unit would reduce the expected distance since the rate of free slot occurrence increases when depth decreases. This option is described and analyzed in Section~\ref{sec:jump}.
	
	Therefore, taking $k=\km+1$  segments in ~\eqref{mean_Jumpless2}, $\CS_1$ with $\lambda_1=\mu_{\km}\lambda$ and depth $\km$,  \ldots, $\CS_{\km+1}$ with $\lambda_{\km+1}=\mu_{0}\lambda$ and depth $0$, i.e. for $i=1,\ldots,\km+1$,$\lambda_i=\mu_{\km-i+1}\lambda \text{ and }   1/\alpha_i= 2^{i-\km-1}L,\quad \text{with } \mu_k=p/2(q/2)^k $
	in fact $2^{i-\km-1}L$ where $L$ is the typical length of a full segment if one wishes to depart from the unit square model, we finally obtain 
	\begin{align}\label{eq:int_mean}
		\mathbb{E}[D(\CS_1,\ldots,\CS_{\km+1})]=
		\sum_{k=1}^{\km}\frac{L}{2^{k}}\prod_{j\geq k}\frac{1}{1+\lambda L 2^{-j}(p/2)(q/2)^j}.
	\end{align}
	
	Let us define the auxiliary function $g$ on the real positive set  by $g(x)=\prod_{j= 1}^\infty\frac{1}{1+x 2^{-j}(q/2)^j}$ and function $f$ by
	\begin{align}\label{eq:deff}
		f(x)=\sum_{k=1}^\infty\frac{L}{2^{k} }g(\alpha^k x),    
	\end{align} 
	then, assuming $\km=\infty$ and $\alpha=q/4$, we obtain 
	\begin{align}\label{eq:f}
		\E[D(\CS_1,\ldots,\CS_{k},\ldots)]=f(\lambda Lp/(2\alpha)).  
	\end{align}  
	
	\begin{lemma}
		The function $g(x)$ behaves like
		$
		x^{\frac{\log x}{2\log\alpha}}\sqrt{x}
		$
		when $x\to\infty$.\end{lemma}
	\begin{proof}
		The proof uses the Mellin transform of the function $\ell(x)=\log g(x)=-\sum_{k=1}^\infty \log(1+\alpha^kx)$, i.e. $\ell^*(s)=\int_0^\infty \ell(x)x^{s-1}dx$. The Mellin transform of the function $\log(1+x)$ is $\int_0^\infty \log(1+x)x^{s-1}dx=\frac{\pi}{s\sin(\pi s)}$ defined for $-1<\Re(s)<0$. By simple algebra, we obtain $\ell^*(s)=-\sum_{k=1}^\infty\alpha^{-k}\frac{\pi}{s\sin(\pi s)}=-\frac{\alpha^{-s}}{1-\alpha^{-s}}\frac{\pi}{s\sin(\pi s)}$ still defined for $-1<\Re(s)<0$ and by Mellin inversion, we have $\ell(x)=\frac{1}{2i\pi}\int \ell^*(s)x^{-s}ds$. This integral has a triple pole at $s=0$ and its residue is $\frac{(\log x)^2}{2\log\alpha}+\frac{\log x}{2}+\frac{\pi^2}{6\log\alpha}+\frac{\log\alpha}{12}$. Extending the integration path to the right of $s=0$,
		\[
		\ell(x)=\frac{(\log x)^2}{2\log\alpha}+\frac{\log x}{2}+\frac{\pi^2}{6\log\alpha}+\frac{\log\alpha}{12}+O\left(\frac{1}{x}\right).
		\]
	\end{proof}
	\begin{theorem}\label{th:no-jump}
		As $\lambda\to 0$, $f(\lambda)\to L$. As  $\lambda\to\infty$ then $f(\lambda)$ behaves like $\lambda^{-1/d_F}$.
	\end{theorem}
	\begin{proof}
		The proof of the first assertion about $\lambda\to 0$ follows immediately by definition of $f$ and $g$. The proof for $\lambda\to\infty$ is more involved and we use again the Mellin transform. 
		The function $g(x)$ decays faster than any negative power of $x$ thus its Mellin transform $g^*(s)=\int_0^\infty g(x)x^{s-1}dx$ exists for all $s$ such that $\Re(s)>0$. 
		
		By definition of $f$, we have
		\begin{align}\label{f*}
			f^*(s)=\int_0^\infty f(x)x^{s-1}dx=\frac{L}{2}\frac{(q/4)^{-s}}{1-\frac{1}{2}(q/4)^{-s}}g^*(s).
		\end{align}
		The main singularity of $\frac{1}{1-\frac{1}{2}(q/4)^{-s}}$ is a pole at $s=\frac{1}{d_F}$ (in fact a countable set of simple poles on the vertical line $\Re(s)=\frac{1}{d_F}$) and by virtue of the inverse Mellin transform, $f(x)=\frac{1}{2 i\pi}\int f^*(s)x^{-s}ds$ (more in~\cite{mellin}), $f(x)$ is of order $x^{-1/d_F}$ (see details and the exact asymptotics  in  Appendix).
	\end{proof}
	
	Since the hyperfractal dimension $d_F$ is always greater than 2, the distance driven before a free slot cannot decay faster than $\lambda^{-1/2}$. There is a kind of paradox here because when $d_F=2$ one would expect that the street popup rates are the same for all streets, and with a constant pop-up rate the driven distance should scale as $1/\lambda$. But in fact it is not because $\lambda$ is equal to the total popup rate and not equal to the popup rate of each street (there are $4 \times 2^{\km}$ streets). When $d_F=2$, which occurs when $p=0$, the popup rate of each street would be zero and the average driven distance would be infinite.
	
	\subsection{Distribution and variance in the jumpless model}
	The analysis of the mean distance to find a parking slot can be extended to the  distribution of this distance. The aim is to obtain a result on the variance of the distance.
	
	Since slot appearances on different streets from independent Poisson processes, denoting by $E_1$ a random variable with exponential distribution with parameter $\lambda_1$,    
	\begin{align}\label{eq:RandomVariableIdentity}
		D(\CS_1,\CS_2,\ldots,\CS_k)=\ind{E_1\leq|\CS_1|}E_1+\ind{E_1>|\CS_1|}(|\CS_2|+ D(\CS_2,\ldots,\CS_k)). 
	\end{align}
	This directly yields
	\begin{align}\label{Laplace}
		\E\left(e^{-s D(\CS_1,\CS_2,\ldots,\CS_k)}\right)=\frac{\lambda_1}{\lambda_1+s}\left(1-e^{-(\lambda_1+s)|\CS_1|}\right)+e^{-(\lambda_1+s)|\CS_1|} \E\left(e^{-s D(\CS_2,\ldots,\CS_k)}\right). 
	\end{align}
	which gives by induction that
	\begin{equation}\label{eq:LaplaceJumpless}
		\E\left(e^{-sD(\CS_1,\CS_2,\ldots,\CS_k)}\right)=\sum_{i=1}^{k}\left(\frac{\lambda_i}{\lambda_i+s}\left(1-e^{-(\lambda_i+s)|\CS_i|}\right)\prod_{j<i}e^{-(\lambda_j +s)|\CS_j|}\right)+\prod_{j\leq k}e^{-(\lambda_j +s)|\CS_j|}.
	\end{equation}
	After some calculations detailed in the appendix, we obtain
	\begin{equation}\label{eq:moment2Jumpless}
		\E(D(\CS_1,\ldots,\CS_k)^2)=2\sum_{i=1}^{k}\frac{1+\lambda_i \sum_{j<i} |\CS_j|-e^{-\lambda_i|\CS_i|}(1+\lambda_i\sum_{j\leq i} |\CS_j|)}{\lambda_i^2} \prod_{j<i}e^{-\lambda_j |\CS_j|}.
	\end{equation}
	
	Integrating on segment lengths $\CS_i$, which are independent and have exponential distribution with mean $1/\alpha_i$, we obtain that
	\begin{equation}\label{eq:intmoment2Jumpless}
		\E(D(\CS_1,\ldots,\CS_k)^2)=2\sum_{i=1}^{k}\frac{1/\alpha_i^2}{(1+\lambda_i/\alpha_i)^2} \prod_{j<i}\frac{1}{1+\lambda_j/\alpha_j}
	\end{equation}
	and, using that $\alpha_i=2^{i-\km-1}L$ and $\lambda_i=\mu_{\km-i+1}=\lambda (p/2)(q/2)^{\km-i+1}$, following the same steps as for equation~\eqref{eq:int_mean}, recalling that $\alpha=q/4$ and $\rho=\lambda L p/2$,
	\begin{align}\label{eq:intmoment2Jumplessv2}
		\E(D(\CS_1,\ldots,\CS_{\km+1})^2)=2 \sum_{k=0}^{\km}\frac{L^2}{4^{k}}\frac{1}{(1+\rho \alpha^k)^2} \prod_{j>k}\frac{1}{1+\rho \alpha^j}
	\end{align}
	Recalling that $g$ is already defined for $f$ in equation~\eqref{eq:deff}, defining   the auxiliary function $h$ on the set of positive real numbers  by $h(x)=\frac{1}{(1+x)^2}$ and $F$ by
	\begin{align}\label{eq:defF}
		F(x)=\sum_{k=1}^{\infty} \frac{L^2}{4^{k}} h(\alpha^k x) g(\alpha^k x),
	\end{align}
	$F$ is a harmonic series. Using equation~\eqref{eq:intmoment2Jumplessv2}, in addition to~\eqref{eq:f} and assuming $\km=+\infty$, 
	\begin{align*}
		\Var(D(\CS_1,\ldots,\CS_{\km}))=F(\rho)-f(\rho/\alpha)^2
	\end{align*}
	
	where $f$ is defined by equation~\eqref{eq:deff}. This leads to the following result. 
	
	\begin{theorem}\label{th:no-jump_Variance}
		When  $\lambda\to\infty$, $\Var(D)$ behaves like $\lambda^{-2/d_F}$.
	\end{theorem}
	The argument follows the same Mellin-transform scheme as in the proof of Theorem~\ref{th:no-jump}. We prove more precisely that, as $\lambda$ tends to $+\infty$,
	\begin{align}
		\Var(D)\sim_{\lambda \rightarrow +\infty} \frac{L^2}{\log \alpha} \left(r^*(2/d_F)-\alpha^{1/d_F}g^*(2/d_F)\right)\rho^{-2/d_F}.
	\end{align}
	
	\section{The jump-over model}\label{sec:jump}
	For the analysis of the jump-over model, the sequence of segments $\CS_1,\ldots,\CS_k$ still has decreasing depths, but the depth is no longer required to decrease by exactly one unit at each step. 
	
	\begin{theorem}
		In the jump-over strategy, the distance to a free parking slot is of order $\Theta(\lambda^{-1/d_F})$.
	\end{theorem}
	\begin{proof}
		Let $k_i$ denote the depth of segment $\CS_i$. 
		Proceeding as in the jumpless case, and integrating over the segment lengths, we obtain the following expression, where the dependence on the depths $k_j$ appears through the corresponding intensities:
		\begin{equation}
			\mathbb{E}[D(\CS_1,\ldots,\CS_{\km+1})]
			=\sum_{i=1}^{\km}
			\prod_{j\ge i}\frac{1}{1+\lambda L (p/2)(q/4)^{k_j}}.
		\end{equation}
		
		If we want to sum over all possible jumps, we must consider the probability $p(k_{i-1}|k_i)$ that the car jumps from level $k_i$ to $k_{i-1}$. In the case of the binary hyperfractal we have $p(k_{i}|k_{i+1})=\frac{1}{2^{k_{i+1}-k_i}}$.
		
		Assuming $k_1=0$, we have $\mathbb{E}[D]=f((p/2)\lambda L)$ with 
		\begin{equation}
			f(x)=\sum_{k=0}^\infty L\frac{1}{1+x\alpha^k}\frac{1}{2^{k+1}}\gamma(x\alpha^{k+1})
		\end{equation}
		with
		\begin{equation}
			\gamma(x)=\sum_{0=k_1<\cdots<k_i<\dots}\prod_{i\ge 1} p(k_i|k_{i+1})\frac{1}{1+x\alpha^{k_{i+1}}}
		\end{equation}
		
		We notice that $\gamma(x)$ does not play the same role as $g(x)$ in the previous section, for the jumpless case. Nevertheless we notice that $\gamma(0)=1$. Using that $p(k_i|k_{i+1})=2^{k_i-k_{i+1}}$ is translation-invariant, we obtain the functional equation:
		\[
		\gamma(x)=\sum_{k=0}^\infty\frac{p(0|k)}{1+x\alpha^k}\gamma(x\alpha^{k+1})
		\]
		Thus defining $j(x)=\frac{1}{1+x/\alpha}\gamma(x)$ we get the equation:
		\[
		(1+x/\alpha)j(x)=\sum_{k=1}^\infty \frac{1}{2^k}j(x\alpha^{k}).
		\]
		The Mellin transform $j^*(s)$, of function $j(x)$, satisfies the equation
		\begin{equation}
			j^*(s)+\frac{1}{\alpha}j^*(s+1)=\sum_{k=1}^\infty\frac{\alpha^{-ks}}{2^k}j^*(s)=\frac{\alpha^{-s}/2}{1-\alpha^{-s}/2}j^*(s).
		\end{equation}
		Equivalently,
		\begin{equation}
			j^*(s+1)=-\alpha\left(1-\frac{\alpha^{-s}/2}{1-\alpha^{-s}/2}\right)j^*(s).
		\end{equation}
		If we have a function $\beta(s)$ such that $\frac{\beta(s+1)}{\beta(s)}=-\alpha$ we have
		\[
		\frac{j^*(s+1)}{\beta(s+1)}=\left(1-\frac{\alpha^{-s}/2}{1-\alpha^{-s}/2}\right)\frac{j^*(s)}{\beta(s)}.
		\]
		A good candidate is $\beta(s)=\frac{\alpha^s}{\sin(\pi s)}$. Therefore 
		\begin{equation}
			j^*(s)=\frac{\alpha^s}{\sin(\pi s)}\prod_{k=1}^\infty \left(1-\frac{\alpha^{k-s}/2}{1-\alpha^{k-s}/2}\right)
		\end{equation}
		up to a multiplicative constant chosen so that $\gamma(0)=1$. The factor is $\pi\prod_{k=1}^\infty\left(1-\frac{\alpha^{k}/2}{1-\alpha^{k}/2}\right)^{-1}$. 
		
		We notice that the only pole of $j^*(s)$ between 0 and 2 is the root of $1-\alpha^{1-s}/2$, namely $\frac{1}{d_F}+1$, since the pole at $s=1$ of $1/\sin(\pi s)$ is canceled by the numerator $1-\frac{\alpha^{1-s}/2}{1-\alpha^{1-s}/2}$. Consequently $j(x)$ behaves in $x^{-1-1/d_F}$ when $x\to\infty$.
		
		We have $f(x)=\sum_{k\ge 1}\frac{1}{2^k}\frac{1}{1+x\alpha^{k-1}}j(x\alpha^k)$. Thus introducing $h(x)=\frac{1}{1+x/\alpha}j(x)$ we have $f(x)=\sum_{k\ge 1}\frac{1}{2^k}h(\alpha^kx)$ and the respective Mellin transforms $f^*(s)$ and $h^*(s)$ satisfy 
		\[
		f^*(s)=\frac{\alpha^{-s}/2}{1-\alpha^{-s}/2}h^*(s).
		\]
		
		But $j(x)=(1+x/\alpha)h(x)$, that is $j^*(s)=h^*(s)+\frac{1}{\alpha}h^*(s+1)$. Thus 
		\[
		h^*(s)=\sum_{k\ge 1}\alpha^k j^*(s-k)
		\]
		entailing a sequence of simple poles located at $\frac{1}{d_F
		}+1+k$ (modulo some imaginary components). Since the relation $f^*(s)=h^*(s)\frac{\alpha^{-s}/2}{1-\alpha^{-s}/2}$ adds only a new pole at $s=\frac{1}{d_F}$ (plus some imaginary parts), thus $f(x)=\Theta(x^{-1/d_F})$.
	\end{proof}
	
	
	\section{Distribution of the number of turns before parking}
	The steering wheel is the most used component in a car. We denote $\bT(\lambda)$ the random variable which counts the number of turns at street crossings before parking.

	In the early phase of the parking search the car progresses on short segments with no available slots. Thus when $\km$ is large there would have been infinitely many turns before the car arrives in segments with significant turnover. Therefore the quantity of interest is $\km-\bT(\lambda)$. We have then when $\km\to\infty$ for the jumpless model:
	
	with $c_k=\prod_{k\le\ell\le \km} g(\alpha^\ell x)=\frac{g(\alpha^k x)}{g(\alpha^{\km} x)}$
	and $c(u)=\sum_{k=0}^{\km}c_k u^{\km-k}$. We have
	\begin{equation}
		\mathbb{E}[u^{\bT(\lambda)}]=\left(1-\frac{1}{u}\right)c(u)+\frac{1}{u}
	\end{equation}
	Thus
	\begin{align*}
		\mathbb{E}[u^{\bT(\lambda)}]=&\left(1-\frac{1}{u}\right)\sum_{k=0}^{\km}\frac{g(\alpha^k x)}{g(\alpha^{\km} x)}u^{\km-k}+\frac{1}{u}\\
		=&\left(1-\frac{1}{u}\right)\sum_{k=0}^{\km}\frac{g(\alpha^k x)-1}{g(\alpha^{\km} x)}u^{\km-k}+\frac{1-u^{\km+1}}{1-u}+\frac{1}{u}\\
		=&\left(1-\frac{1}{u}\right)f_{\km}(x,\frac{1}{u})u^{\km}+\frac{1}{u}\left(1-\frac{1}{g(\alpha^{\km}x)}\right)+\frac{u^{\km}}{g(\alpha^{\km}x)}
	\end{align*}
	with $f_k(x,u)=\sum_{\ell\le k}\frac{g(\alpha^\ell x)}{g(\alpha^k x)}u^\ell$. Thus
	\begin{align*}
		\mathbb{E}[u^{\km-\bT(\lambda)}]=&u^{\km}\mathbb{E}[u^{\bT(\lambda)}]\\
		=&(1-u)f_{\km}(x,u)+\frac{1}{g(\alpha^{\km}x)}+u^{\km+1}\left(1-\frac{1}{g(\alpha^{\km}x)}\right)
	\end{align*}
	Since $g(\alpha^k x)-1=O(\alpha^k)$ thus when $\km\to\infty$ $f_{\km}(x,u)\to f(x,u)$ and $u^{\km+1}(1-\frac{1}{g(\alpha^{\km}x)})\to 0$ as long as $|u|<\frac{1}{\alpha}$. Thus when $|u|<\frac{1}{\alpha}$
	\begin{equation}
		\mathbb{E}[u^{\km-T(\lambda)}]\to\sum_{k=0}^\infty (g\left(x\alpha^k\right)-1)u^{k}(1-u)+1.
		\label{eq-uTm}\end{equation}
	with $x=L(p/2)\lambda$.
	This sum converges uniformly for all $u$ such that $|u|<\frac{1}{\alpha}$ since $g(x)-1=O(x)$ when $x\to 0$, and write:
	\[
	f(x,u)=(1-u)\sum_k u^{k}(g(x\alpha^k)-1)
	\]
	and we have $E[u^{\km-\bT(\lambda)}]=f(x,u)+1$. 
	\begin{theorem}
		In the jumpless model, we have $\mathbb{E}[\km-\bT(\lambda)]=\frac{\log_2\lambda}{d_F}+O(1)$ and $\km-\bT(\lambda)=\frac{\log_2\lambda}{d_F}+O(1)$. The $O(1)$ terms are periodic random functions of $\log\lambda$. The convergence is in moment and in distribution.
	\end{theorem}
	
	\begin{proof}
		We know that $\mathbb{E}[\km-\bT(\lambda)]=\frac{\partial}{\partial u}f(x,1)=-\sum_k(g(x\alpha^k)-1)$.
		The function $g(x)$ has a Mellin transform $g^*(s)$ defined on the strip $(0,+\infty)$, this is also the Mellin transform of $g(x)-1$ but now it is defined on the strip $(-1,0)$ with analytic continuation beyond 0 and on $s=0$ there is a simple pole with residue 1, because $g(x)$ is analytic and $g(0)=1$. The Mellin transform of the generating function of the average number of turns is therefore $-\frac{1}{1-\alpha^{-s}}g^*(s)$. This expression has a double pole on $s=0$ and a sequence of simple poles on the integer multiples of $2i\pi/\log\alpha$. Thus in the inverse Mellin operation the residue of this function with the factor $x^{-s}$ will be $-\frac{\log x}{\log\alpha}$ plus a sum of powers of $x^{2i\pi/\log\alpha}$ which contributes as a periodic function of $\log x$ with period $-\log \alpha$. In passing $\log\alpha=-d_F\log 2$.
		
		For the study of the distribution we consider the Mellin transform of $f(x,u)$, $f^*(s,u)$ which is $\frac{1-u}{1-\alpha^{-s}u}g^*(s)$, defined on $(-1,0)$. Let $u=e^t$ with $t\neq 0$, the function $f^*(s,u)$ has simple poles at $s=\frac{t}{\log\alpha}$ plus integer multiples of $2i\pi/\log\alpha$, all with residue $\frac{1}{\log\alpha}$. There is also the isolated pole of $g^*(s)$ at $s=0$ of residue $1$ which will be canceled by the last term 1 in~(\ref{eq-uTm}). The reverse Mellin operation yields 
		\begin{equation}
			f(x,e^t)=\frac{1}{\log\alpha}x^{-t/\log\alpha}\left(g^*(-t/\log\alpha)+P(x,t)\right)-1
		\end{equation}
		Where $P(x,t)$ is periodic in $\log x$ with period $-\log\alpha$. The last term 1 disappears when we consider $\mathbb{E}[e^{t(\km-\bT(\lambda)}]$. We see that the $\mathbb{E}[e^{t(\km-\bT(\lambda)}e^{-t\mathbb{E}[\km-\bT(\lambda)]}]$ is $O(1)$ and since $t$ is possibly complex, by Lévy and Cauchy theorems $\km-\bT(\lambda)+\frac{\log\lambda}{\log\alpha}$ is $O(1)$ in distribution and in moments.
	\end{proof}
	
	The theorem also holds for the jump-over model. But in this case we have the expression
	\begin{equation}
		\mathbb{E}[u^{\km-\bT(\lambda)}]=(1-u)\sum_{i=1}^{k=\km}\prod_{j\ge i}\frac{u}{1+\lambda L (p/2)(q/4)^{k_j}}
	\end{equation}
	which can be handled in a similar way but we omit the details of the proof.
	
	\section{Generalization of hyperfractal distributions}
	\subsection{Beyond dyadic scale factor}
	The Manhattan construction is a convenient representative of a broader
	class of self-similar hierarchical street geometries.
	A minimal generalisation is obtained by replacing the dyadic scale factor
	and the fixed mass ratios by generic geometric growth parameters:
	at refinement step $n$, segments have a typical length
	$\ell_n = c\, s^n$ with $0<s<1$, and carry a total mass
	$m_n = m_0\, r^n$ with $r>0$, uniformly distributed along
	the segments of that level.
	Eliminating the refinement index $n$ between these two relations yields the logarithmic exponent
	\[
	d_F=\frac{\log r}{\log s}.
	\]
	
	This corresponds to a parametric extension of the Manhattan construction,
	where the hierarchical mechanism remains unchanged. A structurally more general hyperfractal framework is introduced in a later section.
	
	Such constructions also extend to heterogeneous cities composed of several interacting neighbourhoods, each characterised by its own growth parameters (see \cite{DeperleJacquetRandomCities}).
	\subsection{Randomization of pop-up intensity}
	We introduce a random modulation of the pop-up intensity.
	To each street of level $k$ we associate a positive random variable $W_k$,
	modeling local attractiveness variations between streets of the same level.
	We assume that $(W_k)_{k\ge1}$ are independent and identically distributed,
	with $W\ge0$ almost surely.
	\\
	
	Conditionally on $(W_k)$, the popup process on a street of level $k$
	is modeled by a Poisson process with intensity $\lambda_k W_k$.
	Proceeding as in the deterministic case, we obtain
	\[
	\mathbb{E}[D(\lambda)\mid (W_j)]
	=
	\sum_{k\ge1}
	\frac{L}{2^k}
	\prod_{j\ge k}
	\frac{1}{1+\alpha^j x W_j},
	\qquad x=\frac{\lambda L p}{2v}.
	\]
	
	Taking expectations and using independence,
	\[
	\mathbb{E}[D(\lambda)]
	=
	\sum_{k\ge1}
	\frac{L}{2^k}
	\prod_{j\ge k}
	G(\alpha^j x),
	\]
	where
	$
	\displaystyle G(u)=\mathbb{E}\!\left[\frac{1}{1+uW}\right].
	$
	and defining
	$
	g(x)=\displaystyle \prod_{k\ge1}G(\alpha^k x)$ and
	$
	f(x)=\displaystyle \sum_{k\ge1}\frac{L}{2^k}g(\alpha^k x),
	$
	we obtain
	\[
	\mathbb{E}[D(\lambda)] = f(x).
	\]
	
	The function $f$ has the same functional structure as in the deterministic case,
	but with the basic factor $(1+\alpha^k x)^{-1}$ replaced by $G(\alpha^k x)$.
	Its asymptotic behaviour is governed by the analytic structure of
	\[
	L_G(s)=\int_0^\infty \log G(u)\,u^{s-1}\,du,
	\]
	through the Mellin transform of $\log g(x)$.
	
	\begin{theorem}
		Assume $W\ge0$ a.s., $\mathbb{E}[W^2]<\infty$, and that there exists $\beta \in (0,1)$, $c_1, c_2 > 0$
		such that
		\[
		c_1 t^\beta \leq \mathbb{P}(W \leq t) \leq c_2 t^\beta
		\quad \text{when $t$ approaches $0$}.
		\]
		Then
		\[
		f(x)\sim C\,x^{-1/d_F}
		\qquad\text{as } x\to\infty,
		\]
		for some constant $C>0$.
		Moreover $f(x) \xrightarrow[x \to 0]{} L$.
	\end{theorem}
	
	The assumptions on $W$ exclude distributions with excessive mass at $0$, which would correspond to streets disappearing with positive probability. As long as $W$ does not concentrate too strongly near zero, the random modulation only affects the multiplicative constant, while the exponent remains unchanged. \\
	In particular, the large-$x$ asymptotics is still governed solely by the hyperfractal dimension $d_F$, showing that the geometry of the network dominates moderate random fluctuations. \\
	
	Typical examples include Gamma distributions, which naturally describe positive variations of street attractiveness and satisfy the required behaviour near zero.
	
	\begin{proof}
		Step 1: Behaviour near $x=0$.\\
		
		Recall that $G(u)=\mathbb{E}\!\left[(1+uW)^{-1}\right]$ with $W\ge0$ a.s.\ and $\mathbb{E}[W^2]<\infty$. 
		The identity $(1+uW)^{-1}=1-uW+u^2\frac{W^2}{1+uW}$ yields
		$G(u)=1-\mathbb{E}[W]\,u+O(u^2)$ as $u\to0$, hence
		$\log G(u)=-\mathbb{E}[W]\,u+O(u^2)$. 
		Therefore
		\[
		\log g(x)=\sum_{k\ge1}\log G(\alpha^k x)
		=-\mathbb{E}[W]\!\left(\sum_{k\ge1}\alpha^k\right)x+O(x^2)
		=O(x),
		\]
		so $g(x)=1+O(x)$ and
		\[
		f(x)=\sum_{k\ge1}\frac{L}{2^k}g(\alpha^k x)
		=L+O(x),
		\]
		which proves $\lim_{x\to0}f(x)=L$. \\
		
		Step 2: Large-$x$ behaviour of $g$.\\
		Let $\ell(x)=\log g(x)=\sum_{k\ge1}\log G(\alpha^k x)$ and recall that
		\[
		(\log G)^*(s)=\int_0^\infty \log G(u)\,u^{s-1}\,du \quad \text{ for } -1 < \Re(s) < 0
		\]
		By Mellin scaling, one has
		\[
		\ell^*(s)=\mathcal{M}[\ell](s)
		=\frac{\alpha^{-s}}{1-\alpha^{-s}}\,(\log G)^*(s).
		\]
		The assumption $c_1 t^\beta\le\mathbb{P}(W\le t)\le c_2 t^\beta$ near $t=0$
		implies 
		(see appendix for details)
		\begin{equation}\label{eq:asymptotics}
			\log G(u)=-\beta\log u+O(1), \qquad u\to\infty.
		\end{equation}
		Thus $(\log G)^*(s)$ has a double pole at $s=0$; combined with the simple pole of $(1-\alpha^{-s})^{-1}$ at $s=0$, this yields a triple pole of $\ell^*(s)$ at $0$. Mellin inversion gives
		\[
		\ell(x)=-c(\log x)^2+O(\log x), \qquad x\to\infty,
		\]
		for some $c>0$, hence $g(x)=e^{\ell(x)}$ decays faster than any power of $x$. 
		In particular, its Mellin transform
		\[
		g^*(s)=\int_0^\infty g(x)x^{s-1}dx
		\]
		is holomorphic in a half-plane $\Re(s)>0$, and in particular at $s_0:=1/d_F$.
		
		Thus, taking Mellin transforms in
		$f(x)=\sum_{k\ge1}\frac{L}{2^k}g(\alpha^k x)$ yields
		\[
		f^*(s)
		=L\,\frac{\alpha^{-s}/2}{1-\alpha^{-s}/2}\,g^*(s).
		\]
		Since $g^*(s)$ is holomorphic at $s_0$, the rightmost singularities of $f^*(s)$ come from the geometric factor $(1-\alpha^{-s}/2)^{-1}$, whose simple poles lie on the vertical line $\Re(s)=s_0$, in particular at $s=s_0=1/d_F$. 
		By Mellin inversion and the residue theorem, the dominant contribution to $f(x)$ as $x\to\infty$ is given by the residue at $s_0$, yielding
		\[
		f(x)\sim C\,x^{-1/d_F}
		\]
		for some constant $C>0$.
	\end{proof}
	
	\subsection{General framework}
	Here we want to generalize the hyperfractal properties beyond the Poisson hyperfractal pattern. In the following we will use the symbol $\Theta(.)$ to express the fact that a function is bounded from above and from below by another function up to multiplicative constants.
	
	The popup rate on a segment $\CS$ is $\lambda(\CS)=x\mu(\CS)$ for some $x>0$ where the $\mu(\CS)$ is normalized street weight. The densities $\mu(\CS)$ are no longer multiples of powers of some $(q/2)$. Let $y$ be a point on a street and $\nu$ be a real less than 1. We denote $\CS(y,\nu)$ the segment starting at $y$ that connects to the closest intersection with a street with normalized density greater than some given $\nu$. We suppose that the quantity $|\CS(y,\nu)|$ is now the distance to the intersection with a road with normalized density greater than some given $\nu$ ({\it i.e.} $\mu(\CS(y,\nu))>\nu$). We assume that the average $|\CS(y,\nu)|$ uniformly over the starting point $y$ is smaller than some $\Theta\left(\nu^{\frac{1}{d_F-1}}\right)$ where $d_F>2$ is the hyperfractal dimension (for the sequel we will drop the indication of the starting point $y$). It still implies that the network of streets of small pop-up rates tend to be dense when the popup rates tend to zero. 
	
	Enumerating the pop-up rates from the largest to the smallest, we then have, for all $i$: $\lambda_i=\lambda(\CS_i)$
	\begin{align}
		D(\CS_1,\ldots,\CS_k)=\sum_{i=0}^{i=k}\frac{1-e^{-\lambda_i|\CS_i|}}{\lambda_i}\prod_{j>i}e^{-\lambda_j|\CS_j|}
	\end{align}
	For a given sequence of decreasing rates $(\lambda_0,\lambda_1,\ldots,\lambda_k)$ the random variables $|S_i|$ are independent. Thus, conditioning on a fixed sequence of densities, we have
	\begin{align}
		\mathbb{E}[D(\CS_1,\ldots,\CS_k)] =  \sum_{i=0}^{k}\mathbb{E}\left[\frac{1-e^{-\lambda_i|\CS_i|}}{\lambda_i}\right]\prod_{j>i}\mathbb{E}[e^{-\lambda_j|\CS_j|}]\\
		\ge\mathbb{E}\left[\frac{1-e^{-\lambda_i|\CS_i|}}{\lambda_i}\right]\prod_{j>i}e^{-\lambda_j\mathbb{E}[|\CS_j|]},
	\end{align}
	for all $0\leq i \leq k$, due to the convexity of the exponential function.
	
	The algorithm for finding a parking space will be the following. We fix a constant parameter $\rho>1$. The car travelling on a street of popup rate $\lambda_i$ (resp. of normalized weight $\mu_i$) will turn at the first intersection with a street having a density $\lambda_{i-1}$ greater than $\rho\lambda_i$ (resp. $\mu_{i-1}>\rho\mu_i$). 
	Since the normalized weights are now random variables, we can define $P\left(\mu_{i-1}\in[\nu,\nu+d\nu]|\mu_{i}\right)$, denoted $p(\nu|\mu_{i})d\nu$ for brevity.
	
	\begin{align}
		f(x)=&\mathbb{E}[D(\CS_1,\ldots,\CS_k,\ldots)]\ge f_2(x)=\int_0^1 \mathbb{E}\left[\frac{1-e^{-x\mu_1|\CS_1|}}{x\mu_1}\right] p(\mu_1|\mu_2)d\mu_1\nonumber\\
		&\times\int_0^{\mu_1/\rho}e^{-x\mu_2\mathbb{E}[|\CS_2|]}p(\mu_2|\mu_3)d\mu_2\cdots\int_0^{\mu_{i-1}/\rho}e^{-x\mu_i\mathbb{E}[|\CS_i|]}p(\mu_i|\mu_{i+1})d\mu_i\cdots\label{eq-gg}
	\end{align}
	We define $f_2(x)$, by translating the $p(\mu_i|\mu_{i+1})$ terms:
	\begin{align*}
		f_2(x)=&\int_0^1\frac{E[1-e^{x\mu_1|S_1|}]}{x\mu_1}\frac{d\mu_1}{\mu_1}\\
		&\int_0^{\mu_1/\rho}e^{-x\mu_2E[|S_2|]}p(\mu_1|\mu_2)\mu_1\frac{d\mu_2}{\mu_2}\cdots\int_0^{\mu_{i-1}/\rho}e^{-x\mu_iE[|\CS_i|]}p(\mu_{i-1}|\mu_{i})\mu_{i-1}\frac{d\mu_i}{\mu_i}\cdots
	\end{align*}

	Let $\delta=\frac{1}{d_F-1}$. Using the general hyperfractal property the average distance $\mathbb{E}[\lambda_i|S_i|]$ should be $\Theta\left(x(\mu_i)^{\delta+1}\right)$ and also $\mathbb{E}[1-e^{-x\mu_i|S_i|}]=\Theta\left(\eta\left(x(\mu_i)^{\delta+1}\right)\right)$ for some bounded function $\eta$. And for $\mu>\rho\mu_i$  $P\left(\mu_{i-1}>\mu|\mu_i\right)=\Theta\left((\frac{\mu}{\mu_i})^{-\delta}\right)$ and $p(\mu|\mu_i)d\mu=\Theta\left((\frac{\mu}{\mu_i})^{-\delta}\right)\frac{d\mu}{\mu}$.
	
	
	\begin{theorem}
		We have $f(x)=\Omega(x^{-1/d_F})$.
	\end{theorem}
	\begin{proof}
		We have (removing the $\Theta$ symbols which will not affect the order of magnitude)
		\[
		f_2(x)=\int_0^{1}\frac{\eta(x\mu^{\delta+1})}{x\mu^2}\gamma(x\mu^{\delta+1})d\mu
		\]
		with 
		\begin{equation}
			\gamma(x)=\int_0^{1/\rho}p(1|\mu)e^{-x\mu^{\delta+1}}\gamma(x\mu^{\delta+1})\frac{d\mu}{\mu}
			\label{eqgamma}\end{equation}
		We can prove that the bounded function $\gamma(x)$ decays faster than any exponential function. To prove this let $\tgamma$ the Laplace transform of $\gamma(x)$:
		\[
		\tgamma(\omega)=\int_0^\infty e^{\omega x}\gamma(x)dx
		\]
		From~(\ref{eqgamma}) we get
		\begin{equation}
			\tgamma(\omega)=\int_0^{1/\rho}p(1|\mu)\tgamma((\omega-1)\mu^{\delta+1})\frac{d\mu}{\mu}
			\label{eqtgamma}\end{equation}
		The Laplace transform $\tgamma(\omega)$ is {\it a priori} defined for $\Re(\omega<0)$, but with the identity~(\ref{eqtgamma}) the domain of definition can be extended to $\Re(\omega)<1$, and thus can be further extended to $\Re(\omega)<2$, and so on, ultimately extending to the whole complex plane. 
		
		Taking this property, the function $h(x)=\eta(x)\gamma(x)$ decays more than exponentially and has a Mellin transform $h^*(s)$ defined for $\Re(s)>-1$. The Mellin transform of $f_2(x)$, $f_2^*(s)$ exist and satisfies:
		\[
		f_2^*(s)=h^*(s-1)\int_0^{1/\rho}\mu^{-(\delta+1)(s-1)}\frac{d\mu}{\mu^2}=\frac{h^*(s-1)}{(\delta+1)(s-1)+1}\rho^{(\delta+1)(s-1)+1}
		\]
		assuming $(\delta+1)(s-1)+1<0$. It has a simple pole at $s=1-\frac{1}{\delta+1}=\frac{1}{d_F}$, thus $f_2(x)=\Theta(x^{-1/d_F})$ and finally $f(x)=\Omega(x^{-1/d_F})$ because~(\ref{eq-gg}) is an inequality.
	\end{proof}
	
	
	\section{Conclusion and Perspectives}
	
	We have shown that the distance to the first available parking slot in a hyperfractal Manhattan network follows a scaling law of order $\lambda^{-1/d_F}$. The exponent depends only on the large-scale geometric structure of the street network, showing that geometry plays a central role in the search process.
	
	Several extensions can be considered. The linear exploration rule studied here provides a simple baseline, but more flexible search strategies allowing adaptive turns or partial backtracking could be analysed to understand their effect on the scaling constants and on the stability of the exponent. Another direction is to study heterogeneous cities made of districts with different hyperfractal parameters, and to analyse vehicles moving across regions with different effective dimensions. One may also introduce spatial correlations in slot availability, non-Poisson release processes, or time-dependent demand. Finally, considering multiple interacting vehicles would allow the study of congestion effects and competition for parking on hyperfractal networks.
	
	\newpage
	\bibliographystyle{plain}
	\bibliography{references}
	\newpage
	\section*{Appendix}
	\begin{proof}[ Proof of Theorem~\ref{th:no-jump}]
		We detail the asymptotics of $f$ for $\lambda\rightarrow +\infty$. As remarked in the proof, by~\eqref{f*}, the main singularities of $f^*$ are a simple pole at $s=1/d_F$, and a sequence of poles of form $1/d_F$ plus non zero integer multiples of
		$+2i\pi/\log\alpha$ the latter contributing to a periodic term of mean zero.
		\begin{align*}
			f(x)\underset{x \to +\infty}{\sim}-Res(f^*,1/d_F)x^{-1/d_F}(1+P(\log x))
		\end{align*}
		where $P(.)$ is a periodic function of mean zero and period $-\log\alpha$. Using that $(s-1/d_F)/(v(s)-v(1/d_F))\rightarrow 1/v'(1/d_F)$ as $s$ tends to $1/d_F$ with $v(s)=1-\alpha^{-s}/2$ and that, straightforwardly, $v'(1/d_F)=-\log(\alpha)$, 
		\begin{align*}
			Res(f^*,1/d_F)&=\lim_{s\rightarrow 1/d_F}(s-1/d_F)f^*(s)=\lim_{s\rightarrow 1/d_F}(s-1/d_F)\frac{L}{2}\frac{\alpha^{-s}}{1-\alpha^{-s}/2}g^*(s)\\
			&= -\frac{L}{2\log(\alpha)}\alpha^{-1/d_F}g^*(1/d_F)= -\frac{L}{\log(\alpha)}g^*(1/d_F)
		\end{align*}
	\end{proof}
	
	\begin{proof}[ Proof of equation~\eqref{eq:moment2Jumpless}]
		To this end, it suffices to take the second derivative of equation~\eqref{eq:LaplaceJumpless}. Denoting by  $a_i$ and $e_i$ functions of $s$ given by
		\begin{align}
			a_i(s)=\frac{\lambda_i}{\lambda_i+s},\quad  e_i(s)=\prod_{j\leq i}e^{-(\lambda_j+s)|\CS_j|}
		\end{align}
		and also denoting $\sigma_i=\sum_{j\leq i}|\CS_j|$,   equation~\eqref{eq:LaplaceJumpless}, giving the Laplace transform of $D(\CS_1,\ldots,\CS_k)$ $\varphi_k=\E\left(e^{-s D_k}/(\CS_1,\ldots,\CS_k)\right)$, can be rewritten as
		\begin{align*}
			\varphi_k=\sum_{i=1}^{k}a_i(e_{i-1}-e_i)+e_k.
		\end{align*}
		Using that $e'_i=-\sigma_i e_i$, differentiating the previous equation twice,
		\begin{align*}
			\varphi'_k=\sum_{i=1}^{k}a'_i(e_{i-1}-e_i)-\sum_{i=1}^{k}a_i(\sigma_{i-1}e_{i-1}-\sigma_i e_i)-\sigma_k e_k.
		\end{align*} 
		and 
		\begin{align}\label{eq:varphi''}
			\varphi''_k=\sum_{i=1}^{k}a''_i(e_{i-1}-e_i)-2\sum_{i=1}^{k}a'_i(\sigma_{i-1}e_{i-1}-\sigma_i e_i)+\sum_{i=1}^{k}a_i(\sigma_{i-1}^2 e_{i-1}-\sigma_i^2 e_i)+\sigma_k^2 e_k.
		\end{align} 
		Replacing $s$ by $0$ gives the result since  in the third sum in the right-hand side of~\eqref{eq:varphi''},  $a_i(0)=1$ for all $i$ and $\sigma(0)=0$. Thus $\varphi''_k(0)$ is reduced to the first two terms. This concludes the proof.
	\end{proof}
	
	\begin{proof}[ Proof of equation~\eqref{eq:asymptotics}.]
		Given that \[\displaystyle G(u) = \E\left[\frac{1}{1+uW}\right] = \int_0^{+\infty} \frac{1}{1+ut}f_W(t)dt = u\int_0^{+\infty} \frac{F(t)}{(1+ut)^2} dt\] via integration by parts where $F(t) = \mathbb{P}(W \leq t)$. 
		Consider $t_0 > 0$ such as $c_1 t^\beta\leq F(t) \leq c_2 t^\beta$ for $t \in (0, t_0)$. We have, for $u$  large enough, \[G(u) \geq u\int_0^{t_0} \frac{F(t)}{(1+ut)^2} dt \geq c_1u\int_0^{t_0} \frac{t^\beta}{(1+ut)^2} dt = c_1 u^{- \beta} \int_{0}^{ut_0} \frac{x^\beta}{(1+x)^2} dx \geq \int_{0}^{1} \frac{x^\beta}{(1+x)^2} dx \]
		For the majoration, we use split the integral, 
		$G(u) = \displaystyle u\int_0^{t_0} \frac{F(t)}{(1+ut)^2} dt +u\int_{t_0}^{+\infty} \frac{F(t)}{(1+ut)^2} dt$
		with $\displaystyle u\int_0^{t_0} \frac{F(t)}{(1+ut)^2} dt \leq c_2 u^{- \beta} \int_{0}^{ut_0} \frac{x^\beta}{(1+x)^2} dx \leq c_2 u^{- \beta} \int_{0}^{+\infty} \frac{x^\beta}{(1+x)^2} dx$  since the integral $\int_{0}^{+\infty} \frac{x^\beta}{(1+x)^2} dx$ converges for $\beta\in (0,1)$. 
		Moreover, because $F(t) \leq 1$,  \[u\int_{t_0}^{+\infty} \frac{F(t)}{(1+ut)^2} dt \leq u\int_{t_0}^{+\infty} \frac{1}{(1+ut)^2} dt = \frac{1}{1+ut_0} \leq  \frac{1}{ut_0} = O(u^{-\beta})\]
		as $u$ tends to $+\infty$.
	\end{proof}
\end{document}